\documentclass[12pt]{amsart}

\newtheorem{theorem}{Theorem}[section]

\theoremstyle{definition}
\newtheorem{definition}{Definition}[section]
\newtheorem{corollary}{Corollary}[section]

\newtheorem{example}{Example}[section]

\theoremstyle{remark}

\numberwithin{equation}{section}

\begin{document}
\title[Invariant submanifolds of $(LCS)_n$-Manifolds ]{Invariant submanifolds of $(LCS)_n$-Manifolds with respect to quarter symmetric metric connection}
\author[S. K. Hui, L.-I. Piscoran and T. Pal]{SHYAMAL KUMAR HUI, LAURIAN-IOAN PISCORAN and TANUMOY PAL}
\subjclass[2000]{53C15, 53C25.} \keywords{$(LCS)_n$-manifold, invariant submanifold, quarter symmetric metric connection}.
\begin{abstract}
The object of the present paper is to study invariant submanifolds of $(LCS)_n$-manifolds with respect to quarter symmetric metric connection. It is
shown that the mean curvature of an invariant submanifold of $(LCS)_{n}$-manifold with respect to quarter symmetric metric connection and Levi-Civita
connection are equal. An example is constructed to illustrate the results of the paper. We also obtain some equivalent conditions of such notion.
\end{abstract}
\maketitle
\section{Introduction}
\indent It is known that a connection $\nabla$ on a Riemannian manifold $M$ is called a metric connection if there is a Riemannian metric $g$ on $M$ such that $\nabla g=0$, otherwise it is non-metric. In 1924, Friedman and Schouten \cite{FRSC} introduced the notion of semi-symmetric linear connection on a differentiable manifold. In 1932, Hayden \cite{HAYDEN} introduced the idea of metric connection with torsion on a Riemannian manifold. In 1970, Yano \cite{YANO} studied some curvature tensors and conditions for semi-symmetric connections in Riemannian manifolds. In 1975, Golab \cite{GOLAB} defined and studied quarter symmetric linear connection on a differentiable manifold. A linear connection $\overline{\nabla}$ in an n-dimensional Riemannian manifold is said to be a quarter symmetric connection \cite{GOLAB} if torsion tensor $T$ is of the form
\begin{equation}\label{1.1}
T(X,Y)= \overline{\nabla}_XY - \overline{\nabla}_YX - [X,Y] = A(Y) K(X)-A(X)K(Y)
\end{equation}
where $A$ is a 1-form and $K$ is a tensor of type (1,1). If a quarter symmetric linear connection $\overline{\nabla}$ satisfies the condition
\begin{equation*}
(\overline{\nabla}_Xg)(Y,Z)=0
\end{equation*}
for all $X$, $Y$, $Z\in \chi(M)$, where $\chi(M)$ is a Lie algebra of vector fields on the manifold $M$, then $\overline{\nabla}$ is said to be a quarter symmetric metric connection. For a contact metric manifold admitting quarter symmetric connection, we can take $A=\eta$ and $K=\phi$ and hence  (\ref{1.1}) takes in the form
\begin{equation}\label{1.2}
T(X,Y)=\eta(Y)\phi X-\eta(X)\phi Y.
\end{equation}
The relation between Levi-Civita connection $\nabla$ and quarter symmetric metric connection $\overline{\nabla}$
of a contact metric manifold is given by
\begin{equation}
\label{1.3}
\overline{\nabla}_XY=\nabla_XY-\eta(X)\phi Y.
\end{equation}
\indent Quarter symmetric connection also studied by Ali and Nivas \cite{ALNI}, Anitha and Bagewadi (\cite{ANIBA} - \cite{ANIBA2}), De and Uddin \cite{DEUD}, Hui \cite{HUI2}, Mishra and Pandey \cite{MISHRA}, Mukhopadhya et al. \cite{MRB}, Prakasha \cite{PRAKAS}, Rastogi \cite{RASTOGI}, Siddesha and Bagewadi \cite{SIBA}, Yano and Imai \cite{YANOI} and many others.\\
\indent In particular if $\phi X=X$ and $\phi Y=Y$, then quarter symmetric reduces to a semi-symmetric connection \cite{FRSC}.
The semi-symmetric connection is the generalized case of quarter symmetric metric connection
and it is important in the geometry of Riemannian manifolds.\\
\indent In 2003, Shaikh \cite{SHAIKH2} introduced the notion of Lorentzian
concircular structure manifolds (briefly, $(LCS)_n$-manifolds),
with an example, which generalizes the notion of LP-Sasakian
manifolds introduced by Matsumoto \cite{8} and also by Mihai and
Rosca \cite{9}. Then Shaikh and Baishya (\cite{SHAIKH3}, \cite{SHAIKH4})
investigated the applications of $(LCS)_n$-manifolds to the general
theory of relativity and cosmology. The $(LCS)_n$-manifolds is also
studied by Hui \cite{SKH}, Hui and Atceken \cite{HUI1}, Prakasha \cite{PRAKAS2}, Shaikh and his
co-authors (\cite{SHAIKH1} - \cite{SHAIKH9}) and many others.\\
\indent The present paper deals with the study of invariant submanifolds of $(LCS)_{n}$-manifolds with respect
to quarter symmetric metric connection.
Section 2 is concerned with some preliminaries which will be used in the sequel. $(LCS)_{n}$-manifolds with
respect to quarter symmetric metric connection
is studied in section 3. In section 4, we study invariant submanifolds of $(LCS)_n$-manifolds with respect to
quarter symmetric metric connection. It is proved that the mean curvature of an invariant submanifold
with respect to quarter symmetric metric connection and Levi-Civita connections are equal. In this section,
we construct an example of such notion to illustrate the result. Section 5 consists with the study of
recurrent invariant submanifolds of $(LCS)_n$-manifold with respect to quarter symmetric
metric connection. We obtain a necessary and sufficient condition of the second fundamental form of an
invariant submanifold of a $(LCS)_n$-manifold with respect
to quarter symmetric metric connection to be recurrent, $2$-recurrent and generalized $2$-recurrent.
Some equivalent conditions of invariant submanifold of
$(LCS)_n$-manifolds with respect to quarter symmetric metric connection are obtained in this section.
\section{Preliminaries}
The covariant differential of the $p^{th}$ order, $p\geq 1$, of a $(0,k)$-tensor field $T$, $k\geq 1$,
defined on a Riemannian manifold $(\widetilde{ M},g)$ with the Levi-Civita connection $\nabla$ is denoted
by $\nabla^pT$.  According to \cite{ROTER} the tensor $T$ is said to be recurrent, respectively 2-recurrent, if
the following condition holds on $M$
\begin{eqnarray}
\label{2.1}(\nabla T)(X_1,X_2,\cdots,X_k ; X)T(Y_1,Y_2,\cdots,Y_k) &=&\\
\nonumber (\nabla T)(Y_1,Y_2,\cdots,Y_k;X)T(X_1,X_2,\cdots,X_k),
\end{eqnarray}
respectively
\begin{eqnarray}
\label{2.2}(\nabla^2 T)(X_1,X_2,\cdots,X_k ; X,Y)T(Y_1,Y_2,\cdots,Y_k) &=&\\
\nonumber (\nabla^2 T)(Y_1,Y_2,\cdots,Y_k;X,Y)T(X_1,X_2,\cdots,X_k),
\end{eqnarray}
where $ X,Y,X_1,Y_1,\cdots,X_k, Y_k \in T \widetilde{M}$. From (\ref{2.1}) it follows that at a point
$x$ $\in \widetilde{M}$ if the tensor $T$ is non-zero then there exists a unique 1-form $\pi$
respectively, a (0,2) tensor $\psi$ , defined on a neighbourhood $U$ of $x$, such that
\begin{equation}\label{2.3}
\nabla T= T\otimes \pi, \ \ \pi = d(\log \|T\|)
\end{equation}
respectively,
\begin{equation}\label{2.4}
\nabla^2T=T\otimes\psi,
\end{equation}
holds on $U$, where $\|T\|$ denotes the norm of $T$ and $\|T\|^2=g(T,T)$.\\
The tensor is said to be generalized 2-recurrent if
\begin{eqnarray*}
(\nabla^2T)(X_1,X_2,\cdots,X_k;X,Y)-(\nabla T\otimes\pi)(X_1, X_2,\cdots,X_k;X,Y)T(Y_1, Y_2,\cdots,Y_k) =\\
(\nabla^2T)(Y_1, Y_2,\cdots,Y_k;X,Y)-(\nabla T\otimes\pi)(Y_1, Y_2,\cdots,Y_k;X,Y)T(X_1,X_2,.....,X_k)
\end{eqnarray*}
holds on $\widetilde{M}$, where $\phi $ is a 1-form on $\widetilde{M}$. From this it follows that
at a point $x\in \widetilde{M}$ if the tensor $T$ is non-zero, then there exists a unique $(0,2)$-tensor
$\psi$, defined on a neighbourhood $U$ of $x$, such that
\begin{equation}\label{2.5}
\nabla^2T=\nabla T\otimes\pi+T\otimes\psi
\end{equation}
holds on $U$.
\par An $n$-dimensional Lorentzian manifold $\widetilde{M}$ is a smooth connected
paracompact Hausdorff manifold with a Lorentzian metric $g$, that
is, $\widetilde{M}$ admits a  smooth symmetric tensor field $g$ of type (0,2)
such that for each point $p\in \widetilde{M}$, the tensor $g_{p}:T_{p}\widetilde{M}\times
T_{p}\widetilde{M}$ $\rightarrow\mathbb{R}$ is a non-degenerate inner product of
signature $(-,+,\cdots,+)$, where $T_{p}\widetilde{M}$ denotes the tangent
vector space of $\widetilde{M}$ at $p$ and $\mathbb{R}$ is the real number
space. A non-zero vector $v$ $\in T_{p}\widetilde{M}$ is said to be timelike
(resp., non-spacelike, null, spacelike) if it satisfies $g_{p}(v,v)
< 0$ (resp, $\leq $ 0, = 0, $> 0$) \cite{NIL}.
\begin{definition}
In a Lorentzian manifold $(\widetilde{M},g)$ a vector field $P$ defined by
\begin{equation*}
g(X,P)=A(X)
\end{equation*}
for any $X\in\Gamma(T\widetilde{M})$, is said to be a concircular vector field \cite{15} if
\begin{equation*}
(\widetilde{\nabla}_{X}A)(Y)=\alpha \{g(X,Y)+\omega(X)A(Y)\},
\end{equation*}
where $\alpha$ is a non-zero scalar and $\omega$ is a closed 1-form
and $\widetilde{\nabla}$ denotes the operator of covariant
differentiation with respect to the Lorentzian metric $g$.
\end{definition}
Let $\widetilde{M}$ be an $n$-dimensional Lorentzian manifold admitting a unit
timelike concircular vector field $\xi$, called the characteristic
vector field of the manifold. Then we have
\begin{equation}
\label{3.1}
g(\xi, \xi)=-1.
\end{equation}
Since $\xi$ is a unit concircular vector field, it follows that
there exists a non-zero 1-form $\eta$ such that for
\begin{equation}
\label{3.2}
g(X,\xi)=\eta(X),
\end{equation}
the equation of the following form holds
\begin{equation}
\label{3.3}
(\widetilde\nabla _{X}\eta)(Y)=\alpha \{g(X,Y)+\eta(X)\eta(Y)\},
\ \ \ (\alpha\neq 0)
\end{equation}
\begin{equation}
\label{3.4}
\widetilde\nabla _{X}\xi = \alpha \{X +\eta(X)\xi\}, \ \ \ \alpha\neq 0,
\end{equation}
for all vector fields $X$, $Y$, where $\widetilde{\nabla}$ denotes the
operator of covariant differentiation with respect to the Lorentzian
metric $g$ and $\alpha$ is a non-zero scalar function satisfies
\begin{equation}
\label{3.5}
{\widetilde\nabla}_{X}\alpha = (X\alpha) = d\alpha(X) = \rho\eta(X),
\end{equation}
$\rho$ being a certain scalar function given by $\rho=-(\xi\alpha)$.
Let us take
\begin{equation}
\label{3.6}
\phi X=\frac{1}{\alpha}\widetilde\nabla_{X}\xi,
\end{equation}
then from (\ref{3.4}) and (\ref{3.6}) we have
\begin{equation}
\label{3.7} \phi X = X+\eta(X)\xi,
\end{equation}
\begin{equation}
\label{3.8}
g(\phi X,Y) = g(X,\phi Y),
\end{equation}
from which it follows that $\phi$ is a symmetric (1,1) tensor and
called the structure tensor of the manifold. Thus the Lorentzian
manifold $\widetilde{M}$ together with the unit timelike concircular vector
field $\xi$, its associated 1-form $\eta$ and an (1,1) tensor field
$\phi$ is said to be a Lorentzian concircular structure manifold
(briefly, $(LCS)_{n}$-manifold), \cite{SHAIKH2}. Especially, if we take
$\alpha=1$, then we can obtain the LP-Sasakian structure of
Matsumoto \cite{8}. In a $(LCS)_{n}$-manifold $(n>2)$, the following
relations hold \cite{SHAIKH2}:
\begin{equation}
\label{3.9}
\eta(\xi)=-1,\ \ \phi \xi=0,\ \ \ \eta(\phi X)=0,\ \ \
g(\phi X, \phi Y)= g(X,Y)+\eta(X)\eta(Y),
\end{equation}
\begin{equation}
\label{3.10}
\phi^2 X= X+\eta(X)\xi,
\end{equation}
\begin{equation}
\label{3.11}
\widetilde{S}(X,\xi)=(n-1)(\alpha^{2}-\rho)\eta(X),
\end{equation}
\begin{equation}
\label{3.12}
\widetilde{R}(X,Y)\xi=(\alpha^{2}-\rho)[\eta(Y)X-\eta(X)Y],
\end{equation}
\begin{equation}
\label{3.13}
\widetilde{R}(\xi,Y)Z=(\alpha^{2}-\rho)[g(Y,Z)\xi-\eta(Z)Y],
\end{equation}
\begin{equation}
\label{3.14}
(\widetilde{\nabla}_{X}\phi)Y=\alpha\{g(X,Y)\xi+2\eta(X)\eta(Y)\xi+\eta(Y)X\},
\end{equation}
\begin{equation}
\label{3.15}
(X\rho)=d\rho(X)=\beta\eta(X),
\end{equation}
\begin{equation}
\label{3.16}
\widetilde{R}(X,Y)Z =\phi \widetilde{R}(X,Y)Z +(\alpha^{2}-\rho)\{g(Y,Z)\eta(X)-g(X,Z)\eta(Y)\}\xi
\end{equation}
for all $X,\ Y,\ Z\in\Gamma(T\widetilde{M})$ and $\beta = -(\xi\rho)$ is a scalar function,
where $\widetilde{R}$ is the curvature tensor and $\widetilde{S}$ is the Ricci tensor of the manifold.\\
\indent Let $M$ be a submanifold of dimension $m$ of a $(LCS)_n$-manifold $\widetilde{M}$ $(m<n)$ with induced
metric $g$. Also let $\nabla$ and $\nabla^{\perp}$ be the induced
connection on the tangent bundle $TM$ and the normal bundle
$T^{\perp}M$ of $M$ respectively. Then the Gauss and Weingarten
formulae are given by
\begin{equation}\label{3.17}
\widetilde{\nabla}_{X}Y = \nabla_{X}Y + \sigma(X,Y)
\end{equation}
and
\begin{equation}\label{3.18}
\widetilde{\nabla}_{X}V = -A_{V}X + \nabla^{\perp}_{X}V
\end{equation}
for all $X,\ Y \in\Gamma(TM)$ and $V\in\Gamma(T^{\perp}M)$, where $\sigma$
and $A_V$ are second fundamental form and the shape operator
(corresponding to the normal vector field $V$) respectively for the
immersion of $M$ into $\widetilde{M}$. The second fundamental form $\sigma$ and the
shape operator $A_V$ are related by \cite{YANO3}
\begin{equation}\label{3.19}
g(\sigma(X,Y),V) = g(A_{V}X,Y),
\end{equation}
for any $X,\ Y \in\Gamma(TM)$ and $V\in\Gamma(T^{\perp}M)$. We note that $\sigma(X,Y)$ is bilinear and since
$\nabla_{fX}Y = f\nabla_{X}Y$ for any smooth function $f$ on a manifold, we have
\begin{equation}\label{3.20}
\sigma(fX,Y) = f\sigma(X,Y).
\end{equation}
For the second fundamental form $\sigma$, the first and second covariant derivatives of $\sigma$ are defined by
\begin{equation}\label{3.21}
(\widetilde{\nabla}_{X}\sigma)(Y,Z) = \nabla_{X}^\perp (\sigma(Y,Z)) - \sigma(\nabla_{X}Y,Z) - \sigma(Y,\nabla_{X}Z)
\end{equation}
and
\begin{eqnarray}
\label{3.22}(\widetilde{\nabla}^2\sigma)(Z,W,X,Y)&=&(\widetilde{\nabla}_X\widetilde{\nabla}_Y\sigma)(Z,W)\\
\nonumber&=&\nabla_X^\bot((\widetilde{\nabla}_Y\sigma)(Z,W)-(\widetilde{\nabla}_Y\sigma)(\nabla_XZ,W)\\
\nonumber&&-(\widetilde{\nabla}_X\sigma)(Z,\nabla_YW)-(\widetilde{\nabla}_{\nabla_XY}\sigma)(Z,W)
\end{eqnarray}
for any vector fields $X$, $Y$, $Z$ tangent to $M$. Then $\widetilde{\nabla}\sigma$ is a normal bundle valued tensor
of type (0,3) and is called the third  fundamental form of $M$, $\widetilde{\nabla}$ is called the Vander-Waerden-Bortolotti connection
of $\widetilde{M}$, i.e. $\widetilde{\nabla}$ is the connection in $TM\oplus T^\perp M$ built with $\nabla$ and $\nabla^\perp$.
If $\widetilde{\nabla}\sigma = 0$, then $M$ is said to have parallel second fundamental form \cite{YANO3}.\\
\indent The mean curvature vector $H$ on $M$ is given by
\begin{equation*}
H=\frac{1}{m}\displaystyle\sum_{i=1}^{m}\sigma(e_i,e_i)
\end{equation*}
where $\{e_1,e_2,\cdots ,e_m\}$ is a local orthonormal frame of vector fields on $M$.\\
\noindent A submanifold $M$ of a $(LCS)_n$-manifold $\widetilde{M}$ is said to be totally umbilical if
\begin{equation}\label{3.28}
\sigma(X,Y)=g(X,Y)H,
\end{equation}
for any vector fields $X$, $Y$ $\in$ $TM$. Moreover if $\sigma(X,Y)=0$ for all $X,\
Y\ \in TM$, then $M$ is said to be totally geodesic and if $H=0$ then $M$ is minimal in $\widetilde{M}$.
\indent For a $(0,l)$ tensor field $T$, $l\geq 1$, and a symmetric $(0,2)$ tensor field $B$, we have
\begin{eqnarray}\label{3.23}
&&Q(B,T)(X_1,\cdots,X_l;X,Y)\\
\nonumber&&= -T((X\wedge_{B}Y)X_1,X_2,\cdots,X_l)-\cdots - T(X_1,\cdots,X_{l-1},(X\wedge_{B}Y)X_{K}),
\end{eqnarray}
where
\begin{equation}\label{3.24}
(X\wedge_{B}Y)Z = B(Y,Z)X - B(X,Z)Y.
\end{equation}
Putting $T = \sigma$ and $B = g$ or $B = S$, we obtain $Q(g,\sigma)$ and $Q(S,\sigma)$ respectively.\\
\indent An immersion is said to be pseudo parallel if
\begin{equation}\label{3.25}
\widetilde{R}(X,Y)\cdot \sigma = (\widetilde{\nabla}_X\widetilde{\nabla}_Y-\widetilde{\nabla}_Y\widetilde{\nabla}_X-\widetilde{\nabla}_{[X,Y]})\sigma=L_1 Q(g,\sigma)
\end{equation}
for all vector fields $X$, $Y$ tangent to $M$ \cite{DESZCZ}. In particular, if $L_1 = 0$ then the manifold is said to be semiparallel.
Again the submanifold $M$ of a $(LCS)_n$-manifold $\widetilde{M}$ is said to be Ricci generalized pseudoparallel \cite{DESZCZ} if its second fundamental form $\sigma$ satisfies
\begin{equation}\label{3.26}
\widetilde{R}(X,Y)\cdot \sigma=L_2 Q(S,\sigma).
\end{equation}
Also the second fundamental form $\sigma$ of submanifold $M$ of a $(LCS)_n$-manifold $\widetilde{M}$ is said to be $\eta$-parallel \cite{YANO3} if
\begin{equation}\label{3.27}
(\nabla_X\sigma)(\phi Y,\phi Z)=0
\end{equation}
for all vector fields $X$, $Y$ and $Z$ tangent to $M$.

\begin{definition}
\cite{bejancu} A submanifold $M$ of a $(LCS)_n$-manifold $\widetilde{M}$ is said to be invariant if
the structure vector field $\xi$ is tangent to $M$ at every point of $M$ and $\phi X$ is tangent to $M$ for any
vector field $X$ tangent to $M$ at every point of $M$, that is $\phi(TM)\subset TM$ at every point of $M$.
\end{definition}
From the Gauss and Weingarten formulae we obtain
\begin{equation}\label{3.29}
\widetilde{R}(X,Y)Z = R(X,Y)Z + A_{\sigma(X,Z)}Y - A_{\sigma(Y,Z)}X,
\end{equation}
where $\widetilde{R}(X,Y)Z$ denotes the tangential part of the curvature tensor of the submanifold.\\
Now we have from tensor algebra
\begin{equation}\label{3.30}
(\widetilde{R}(X,Y)\cdot \sigma)(Z,U) = R^\perp(X,Y)\sigma(Z,U) - \sigma(R(X,Y)Z,U) - \sigma(Z,R(X,Y)U)
\end{equation}
for all vector fields $X$, $Y$, $Z$ and $U$, where
\begin{equation*}
R^\perp(X,Y) = [\nabla_{X}^\perp,\nabla_{Y}^\perp] - \nabla_{[X,Y]}^\perp.
\end{equation*}
In an invariant submanifold $M$ of a $(LCS)_n$-manifold $\widetilde{M}$, the following relations hold \cite{SHAIKH9}:
\begin{equation}\label{3.31}
\nabla_{X}\xi = \alpha \phi X,
\end{equation}
\begin{equation}\label{3.32}
\sigma(X,\xi) = 0,
\end{equation}
\begin{equation}\label{3.33}
R(X,Y)\xi=(\alpha^{2}-\rho)[\eta(Y)X-\eta(X)Y],
\end{equation}
\begin{equation}\label{3.34}
S(X,\xi)=(n-1)(\alpha^{2}-\rho)\eta(X), \ \ \text{i.e., } Q\xi = (n-1)(\alpha^{2}-\rho)\xi,
\end{equation}
\begin{equation}\label{3.35}
(\nabla_{X}\phi)(Y)=\alpha \{g(X,Y)\xi+2\eta(X)\eta(Y)\xi+\eta(Y)X\},
\end{equation}
\begin{equation}\label{3.36}
\sigma(X,\phi Y)= \phi \sigma(X,Y) =\sigma(\phi X, Y) = \sigma(X,Y)=\sigma(\phi X,\phi Y).
\end{equation}
\section{$(LCS)_n$-manifold with respect to quarter symmetric metric connection}
Let $\overline{\widetilde{\nabla}}$ be a linear connection and $\widetilde{\nabla}$ be the Levi-Civita connection of a $(LCS)_n$-manifold
$\widetilde{M}$ such that
\begin{equation}\label{4.1}
\overline{\widetilde{\nabla}}_XY=\widetilde{\nabla}_XY+U(X,Y),
\end{equation}
where $U$ is a (1,1) type tensor and $X,\ Y\in \Gamma(T\widetilde{M})$.
For $\overline{\widetilde{\nabla}}$ to be a quarter symmetric metric connection on $\widetilde{M}$, we have
\begin{equation}\label{4.2}
U(X,Y)=\frac{1}{2}[T(X,Y)+T^\prime(X,Y)+ T^\prime(Y,X)],
\end{equation}
where
\begin{equation}\label{4.3}
g(T^\prime(X,Y),Z)=g(T(Z,X),Y).
\end{equation}
From (\ref{1.2}) and (\ref{4.3}) we get
\begin{equation}\label{4.4}
T^\prime(X,Y)=\eta(X)\phi Y-g(Y,\phi X)\xi.
\end{equation}
So,
\begin{equation}\label{4.5}
U(X,Y)=\eta(Y)\phi X-g(Y,\phi X)\xi.
\end{equation}
Therefore a quarter symmetric metric connection $\overline{\widetilde{\nabla}}$ in a $(LCS)_n$-manifold $\widetilde{M}$ is given by
\begin{equation}\label{4.6}
\overline{\widetilde{\nabla}}_XY=\widetilde{\nabla}_XY+ \eta(Y)\phi X-g(\phi X,Y)\xi.
\end{equation}
\par Let $\overline{\widetilde{R}}$ and $\widetilde{R}$ be the curvature tensors of a $(LCS)_n$-manifold $\widetilde{M}$ with respect to the quarter symmetric metric connection $\overline{\widetilde{\nabla}}$ and the Levi-Civita connection $\widetilde{\nabla}$ respectively. Then we have
\begin{eqnarray}\label{4.7}
\overline{\widetilde{R}}(X,Y)Z&=& \widetilde{R}(X,Y)Z+(2\alpha-1)\left[g(\phi X,Z)\phi Y-g(\phi Y,Z)\phi X \right]  \\
 \nonumber &&+\alpha\left[\eta(Y)X-\eta(X)Y\right]\eta(Z) \\
\nonumber &&  +\alpha \left[ g(Y,Z)\eta(X)-g(X,Z)\eta(Y)\right]\xi,
\end{eqnarray}
where $\overline{\widetilde{R}}(X,Y)Z=\overline{\widetilde{\nabla}}_X \overline{\widetilde{\nabla}}_Y Z-\overline{\widetilde{\nabla}}_Y \overline{\widetilde{\nabla}}_XZ-\overline{\widetilde{\nabla}}_{[X,Y]}Z$
and $X,\ Y,\ Z\in \chi(\widetilde{M})$.
\par By suitable contraction we have from (\ref{4.7}) that
\begin{eqnarray}\label{4.8}
  \overline{\widetilde{S}}(Y,Z) &=& \widetilde{S}(Y,Z)+(\alpha-1)g(Y,Z)+(n\alpha-1)\eta(Y)\eta(Z) \\
  \nonumber&&-(2\alpha-1)a g(\phi Y,Z),
\end{eqnarray}
where $\overline{\widetilde{S}}$ and $\widetilde{S}$ are the Ricci tensors of $\widetilde{M}$ with respect to $\overline{\widetilde{\nabla}}$ and $\widetilde{\nabla}$ respectively and $a=trace\phi$.
Also we have
\begin{eqnarray}\label{4.9}
  \overline{\widetilde{r}} &=& \widetilde{r}-(2\alpha-1)a^2-(n-1),
\end{eqnarray}
where $\overline{\widetilde{r}}$ and $\widetilde{r}$ are the scalar curvature of $\widetilde{M}$ with respect to $\overline{\widetilde{\nabla}}$ and $\widetilde{\nabla}$ respectively. The relation (\ref{4.8}) can be written as
\begin{eqnarray}\label{4.10}
  \overline{\widetilde{Q}}Y &=& \widetilde{Q} Y+(\alpha-1)Y+(n\alpha-1)\eta(Y)\xi -(2\alpha-1)a \phi Y,
\end{eqnarray}
where $\overline{\widetilde{Q}}$ and $\widetilde{Q}$ are the Ricci operator of $\widetilde{M}$ with respect to the connections $\overline{\widetilde{\nabla}}$ and $\widetilde{\nabla}$ respectively such that
$g(\overline{\widetilde{Q}} X, Y)=\overline{\widetilde{S}}(X,Y)$ and $g(\widetilde{Q} X, Y)=\widetilde{S}(X, Y)$ for all $X, Y\in (\chi{M})$.
Hence we get
\begin{eqnarray}
\label{4.11} \overline{\widetilde{R}}(X,Y)\xi &=& \widetilde{R}(X,Y)\xi -\alpha[\eta(Y)X-\eta(X)Y]\\
\nonumber &=& (\alpha^2-\alpha-\rho) \left[\eta(Y)X-\eta(X)Y\right],\\
\label{4.12} \overline{\widetilde{R}}(X,\xi)Y &=& (\alpha^2-\alpha-\rho) \left[\eta(Y)X-g(X,Y)\xi\right] \\
\nonumber&=&-\overline{\widetilde{R}}(\xi,X)Y,\\
\label{4.13}  \overline{\widetilde{S}}(Y,\xi) &=&(n-1)(\alpha^2-\alpha-\rho)\eta(Y),\\
\label{4.14}  \overline{\widetilde{Q}}\xi &=& (n-1)(\alpha^2-\alpha-\rho)\xi.
\end{eqnarray}
\begin{theorem}
In a $(LCS)_n$-manifold $\widetilde{M}$ with respect to   quarter symmetric metric connection $\overline{\widetilde{\nabla}}$ we have
\begin{eqnarray*}
\overline{\widetilde{R}}(X,Y)Z+\overline{\widetilde{R}}(Y,Z)X+\overline{\widetilde{R}}(Z,X)Y &=& 0,\\
\overline{\widetilde{R}}(X,Y,Z,U)+\overline{\widetilde{R}}(Y,X,Z,U) &=& 0,\\
\overline{\widetilde{R}}(X,Y,Z,U)+\overline{\widetilde{R}}(X,Y,U,Z) &=& 0,\\
\overline{\widetilde{R}}(X,Y,Z,U)-\overline{\widetilde{R}}(Z,U,X,Y) &=& 0.
\end{eqnarray*}
\end{theorem}
\section{Invariant submanifolds of $(LCS)_n$-Manifolds with respect to quarter symmetric metric connection}
Let $M$ be an invariant submanifold of a $(LCS)_n$-manifolod $\widetilde{M}$ with the Levi-Civita connection $\widetilde{\nabla}$ and quarter symmetric metric  connection $\overline{\widetilde{\nabla}}$. Let $\nabla$ be the induced connection on $M$ from the connection $\widetilde{\nabla}$ and $\overline{\nabla}$ be the induced connection on $M$ from the connection $\overline{\widetilde{\nabla}}$.\par Let $\sigma$ and $\overline{\sigma}$ be second fundamental form with respect to Levi-Civita connection and quarter symmetric metric connection respectively.\par Then we have
\begin{eqnarray}
\label{5.2}\overline{\widetilde{\nabla}}_XY &=& \overline{\nabla}_XY+\overline{\sigma}(X,Y).
\end{eqnarray}
From (\ref{4.6}) and (\ref{5.2}) we get
\begin{eqnarray}\label{5.3a}
\overline{\nabla}_XY +\overline{\sigma}(X,Y)&=& \widetilde{\nabla}_XY+\eta(Y)\phi X-g(\phi X,Y)\xi \\
&=& \nabla_XY+\sigma(X,Y)+\eta(Y)\phi X-g(\phi X,Y)\xi.
\end{eqnarray}
Since $M$ is invariant so $\phi X$, \ $\xi \in TM$ and therefore equating tangential and normal parts we get
\begin{eqnarray}
\label{5.3}  \overline{\nabla}_XY &=& \nabla_XY +\eta(Y)\phi X-g(\phi X,Y)\xi,\\
\label{5.4} \overline{\sigma}(X,Y) &=& \sigma(X,Y).
\end{eqnarray}
i.e., the second fundamental form with respect to $\overline{\widetilde{\nabla}}$ and $\widetilde{\nabla}$ are same.\\
Also from (\ref{4.6}) and (\ref{5.3}) we can say that the invariant submanifold $M$ admits quarter symmetric metric connection.\\
Hence we get the following:
\begin{theorem}
Let $M$ be an invariant submanifold of a $(LCS)_n$-manifolod $\widetilde{M}$ with the Levi-Civita connection
$\widetilde{\nabla}$ and quarter symmetric metric connection $\overline{\widetilde{\nabla}}$ and $\nabla$ be the
induced connection on $M$ from the connection $\widetilde{\nabla}$ and $\overline{\nabla}$ be the induced
connection on $M$ from the connection $\overline{\widetilde{\nabla}}$. If $\sigma$ and $\overline{\sigma}$ are the
second fundamental forms with respect to the Levi-Civita connection and quarter symmetric metric connection
respectively then\\
\emph{(i)} $M$ admits quarter symmetric metric connection.\\
\emph{(ii)} The second fundamental forms with respect to $\widetilde{\nabla}$ and $\overline{\widetilde{\nabla}}$ are equal.
\end{theorem}
We now prove the following:
\begin{theorem}
The mean curvature of an invariant submanifold $M$ of a $(LCS)_n$-manifold $\widetilde{M}$ remains invariant under quarter symmetric metric connection.
\end{theorem}
\begin{proof}
Let $\{e_1,e_2,\cdots,e_m\}$ be an orthonormal basis of $TM$. Therefore, from (\ref{5.4}) we get
\begin{equation*}
\overline{\sigma}(e_i,e_i)=\sigma(e_i,e_i).
\end{equation*}
Summing up for $i=1,2,\cdots,m$ and dividing by $m$ we get the desired result.
\end{proof}
\begin{corollary}
The invariant submanifold $M$ of $(LCS)_n$-manifold $\widetilde{M}$ with respect to quarter symmetric
metric connection is minimal with respect to the quarter symmetric metric  connection if and only if
it is minimal with respect to the Levi-Civita connection.
\end{corollary}
\begin{corollary}
The invariant submanifold $M$ of $(LCS)_n$-manifold $\widetilde{M}$ is totally umbilical with respect to the quarter symmetric metric connection if and only if it is totally umbilical with respect to the Levi-Civita connection.
\end{corollary}
We now construct an example:
\begin{example}
Let us consider the $5$-dimensional manifold $\widetilde{M} = \{(x,y,z,u,v)\in \mathbb{R}^5: (x,y,z,u,v)\neq (0,0,0,0,0)\}$, where $(x,y,z,u,v)$ are the standard coordinates in $\mathbb{R}^5$. The vector fields \\$e_1= e^{-z}\frac{\partial}{\partial x}$,\ \ \ \ $e_2 = e^{-z}\frac{\partial}{\partial y}$, \ \ \ \ $e_3 = e^{-2z}\frac{\partial}{\partial z}$,\ \ \ \ $e_4 =e^{-z}\frac{\partial}{\partial u}$,\ \ \ \ $e_5 = e^{-z}\frac{\partial}{\partial v}$\\ are linearly independent at each point of $\widetilde{M}$.\\
Let $\widetilde{g}$ be the metric defined by
\begin{eqnarray*}
\widetilde{g}(e_i,e_j)=\left\{\begin{array}{ll}
1,\, for \,\, i=j\neq3,\\
0,\, for \,\, i\neq j,\\
-1,\ for \, i=j=3.
\end{array}\right.
\end{eqnarray*}
Here $i$ and $j$ runs over 1 to 5.\\ Let $\eta$ be the 1-form defined by $\eta(Z)=\widetilde{g}(Z,e_3)$, for any vector field $Z\in \chi(\widetilde{M})$. Let $\phi$ be the (1,1) tensor field defined by $\phi e_1= e_1$,\ \ \ \
$\phi e_2 = e_2$,\ \ \ \  $\phi e_3=0$,\ \ \ \ $\phi e_4=  e_4$,\ \ \ \  $\phi e_5=e_5$. Then using the linearity property of $\phi$ and $\widetilde{g}$ we have   $\eta(e_3)=-1,$\ \ \ \  $\phi^2 U=U+\eta(U)\xi$\ \ and\ \ $\widetilde{g}(\phi U,\phi V)=\widetilde{g}(U,V)+\eta(U)\eta(V)$, for every $U,\ V\in \chi(\widetilde{M})$.Thus for $e_3=\xi$, $(\phi,\xi,\eta,g)$ defines a Lorentzian paracontact structure on $\widetilde{M}$ .Let $\widetilde{\nabla}$ be the Levi-Civita connection on $\widetilde{M}$ with respect to the metric $\widetilde{g}$. Then we have $[e_1, e_3]=e^{-2z}e_1$,\ $[e_2, e_3]=e^{-2z}e_2$,\ \
$[e_4, e_3]=e^{-2z}e_4$,\ \ $[e_5, e_3]=e^{-2z}e_5$.\\
Now, using Koszul's formula for $\widetilde{g}$, it can be calculated that $\widetilde{\nabla}_{e_1}e_1=e^{-2z}e_3$,
 \ \  $\widetilde{\nabla}_{e_1}e_3=e^{-2z}e_1$,\ \ \ \ $\widetilde{\nabla}_{e_2}e_2=e^{-2z}e_3$,\ \ \ \ $\widetilde{\nabla}_{e_2}e_3=e^{-2z}e_2$,\ \
 $\widetilde{\nabla}_{e_4}e_3=e^{-2z}e_4$,\ \ \ \ $\widetilde{\nabla}_{e_4}e_4=e^{-2z}e_3$,\ \ \ \ $\widetilde{\nabla}_{e_5}e_3=e^{-2z}e_5$,
\ \ and \ \  $\widetilde{\nabla}_{e_5}e_5=e^{-2z}e_3$.\\
and rest of the terms are zero.\\
Since $\{e_1,e_2,e_3,e_4,e_5\}$ is a frame field, then any vector field $X,Y\in T\widetilde{M}$ can be written as
\begin{eqnarray*}
X &=& x_1e_1+x_2e_2+x_3e_3+x_4e_4+x_5e_5,\\
Y &=& y_1e_1+y_2e_2+y_3e_3+y_4e_4+y_5e_5.
\end{eqnarray*}
where $x_i,y_i \in \mathbb{R},\ \  i= 1,2,3,4,5$ such that$$  x_1y_1+x_2y_2-x_3y_3+x_4y_4+x_5y_5 \neq 0$$ and hence
\begin{equation}\label{5.13}
\widetilde{g}(X,Y)=\left(x_1y_1+x_2y_2-x_3y_3+x_4y_4+x_5y_5 \right).
\end{equation}
Therefore,
\begin{eqnarray}
\label{5.13a}\widetilde{\nabla}_XY&=&e^{-2z}[x_1y_3e_1+x_2y_3e_2+(x_1y_1+x_2y_2+x_4y_4+x_5y_5)e_3\\
\nonumber&&+x_4y_3e_4+x_5y_3e_5]
\end{eqnarray}
and
\begin{eqnarray}
\label{5.14}
\overline{\widetilde{\nabla}}_XY&=&e^{-2z}[x_1y_3e_1+x_2y_3e_2+(x_1y_1+x_2y_2+x_4y_4+x_5y_5)e_3\\
\nonumber&&+x_4y_3e_4+x_5y_3e_5]-y_3\left(x_1e_1+x_2e_2+x_4e_4+x_5e_5\right)\\
\nonumber&&-\left(x_1y_1+x_2y_2+x_4y_4+x_5y_5\right)e_3
\end{eqnarray}
From the above it can be easily seen that $(\phi, \xi, \eta,g)$ is a $(LCS)_5$ structure 
on $\widetilde{M}$ with $\alpha= e^{-2z}\neq0$ such that $X(\alpha)=\rho \eta(X)$ where $\rho=2e^{-4z}$. 
Also $\overline{\widetilde{\nabla}}_X\widetilde{g}=0$. Thus in an $(LCS)_5$-manifold quarter symmetric metric connection is given by (\ref{5.14}).\\
Let $f$ be an isometric immersion from $M$ to $\widetilde{M}$ defined by $f(x,y,z)=(x,y,z,0,0)$.
Let $M = \{(x,y,z)\in \mathbb{R}^3: (x,y,z)\neq (0,0,0)\}$, where $(x,y,z)$ are the standard coordinates
in $\mathbb{R}^3$. The vector fields \\$e_1= e^{-z}\frac{\partial}{\partial x}$,\ \ \ \ $e_2 = e^{-z}\frac{\partial}{\partial y}$, \ \ \ \ $e_3 = e^{-2z}\frac{\partial}{\partial z}$ are linearly independent at each point of $M$.\\
Let  $g$ be the induced metric defined by
\begin{eqnarray*}
g(e_i,e_j)=\left\{\begin{array}{ll}
1,\, for \, i=j\neq3,\\
0,\, for \,\, i\neq j,\\
-1,\ for \, i=j=3.
\end{array}\right.
\end{eqnarray*}
Here $i$ and $j$ runs over 1 to 3.\\
Let $\nabla$ be the Levi-Civita connection on $M$ with respect to the metric $g$. Then we have $[e_1, e_3]=e^{-2z}e_1$,\ $[e_2, e_3]=e^{-2z}e_2$.
Now, using Koszul's formula for $g$, it can be calculated that$\nabla_{e_1}e_1=e^{-2z}e_3$,\ \ $\nabla_{e_1}e_3=e^{-2z}e_1$,\ \ \ \ $\nabla_{e_2}e_2=e^{-2z}e_3$,\ \ \ \ $\nabla_{e_2}e_3=e^{-2z}e_2$, and rest of the terms are zero.
Clearly $\{e_4, e_5\}$ is the frame field for the normal bundle $T^\bot M$. If we take $Z\in TM$ then $\phi Z\in TM$ and
therefore $M$ is an invariant submanifold of $\widetilde{M}$. If we take $X,\ Y\in TM$ then we can express them as
\begin{eqnarray*}
X &=& x_1e_1+x_2e_2+x_3e_3,\\
Y &=& y_1e_1+y_2e_2+y_3e_3.
\end{eqnarray*}
Therefore
\begin{equation*}
\nabla_XY=e^{-2z}[x_1y_3e_1+x_2y_3e_2+(x_1y_1+x_2y_2)e_3]\\.
\end{equation*}
Therefore the second fundamental form
\begin{eqnarray}
\sigma(X,Y)=e^{-2z}(x_4y_3e_4+x_5y_3e_5).
\end{eqnarray}
Now, for $X,Y\in TM$ we have
\begin{eqnarray*}
\overline{\widetilde{\nabla}}_XY&=&e^{-2z}[x_1y_3e_1+x_2y_3e_2+(x_1y_1+x_2y_2)e_3+x_4y_3e_4+x_5y_3e_5]-y_3\left(x_1e_1+x_2e_2\right)\\
\nonumber&&-\left(x_1y_1+x_2y_2\right)e_3.
\end{eqnarray*}
Tangential part of $\overline{\widetilde{\nabla}}_XY$ is given by
\begin{eqnarray*}
\overline{\nabla}_XY &=&e^{-2z}[x_1y_3e_1+x_2y_3e_2+(x_1y_1+x_2y_2)e_3]-y_3\left(x_1e_1+x_2e_2\right)\\
\nonumber&&-\left(x_1y_1+x_2y_2\right)e_3.\\
\nonumber&&=\nabla_XY+\eta(Y)\phi X-g(\phi X,Y)\xi,
\end{eqnarray*}
which means $M$ admits quarter symmetric metric connection and the normal part of $\overline{\widetilde{\nabla}}_XY$ is given by
\begin{eqnarray*}
\overline{\sigma}(X,Y)&=&e^{-2z}(x_4y_3e_4+x_5y_3e_5)\\
&=& \sigma(X,Y),
\end{eqnarray*}
which implies that the second fundamental forms with respect to $\overline{\widetilde{\nabla}}$ and $\widetilde{\nabla}$ are equal.
Also $H$ corresponds to Levi-Civita connection as well as quarter symmetric metric connection is zero.
Therefore, $M$ is totally umbilical with respect to both the connections.
\end{example}
Thus Theorem 4.1, Theorem 4.2, Corollary 4.1 and Corollary 4.2 are verified.\\
From (\ref{3.32}) we get
\begin{equation}\label{5.6}
\overline{\sigma}(X,\xi)=0.
\end{equation}
Also from (\ref{4.6}) we have for $V\in T^\bot M$
\begin{eqnarray*}
\nonumber\overline{\widetilde{\nabla}}_XV &=& \widetilde{\nabla}_XV+\eta(V)\phi X-g(\phi X,V)\xi \\
\nonumber&=& \widetilde{\nabla}_XV
\end{eqnarray*}
From (\ref{3.18})we get
\begin{equation}\label{5.6a}
\overline{\widetilde{\nabla}}_XV =-A_VX+\nabla_X^\bot V
\end{equation}
\begin{theorem}
Let $M$ be an invariant submanifold of a $(LCS)_n$-manifold $\widetilde{M}$ with respect to a quarter symmetric metric connection. Then Gauss and Weingarten formulae with respect to quarter symmetric metric connection are given by
\begin{eqnarray}
\nonumber  tan(\overline{\widetilde{R}}(X,Y)Z) &=& \widetilde{R}(X,Y)Z+\eta(\nabla_YZ)\phi X-\eta(\nabla _X Z)\phi Y-g(\phi X,\nabla_YZ)\xi \\
\label{5.9}&& +g(\phi Y, \nabla_XZ)\xi-A_{\sigma(Y,Z)}X+A_{\sigma(X,Z)}Y+\nabla_X\eta (Z)\phi Y\\
\nonumber&&-\nabla_Y\eta (Z)\phi X+\eta(Z)\nabla_X\phi Z-\eta(Z)\nabla_Y\phi Z-\nabla_Xg(\phi Y,Z)\xi \\
\nonumber&&+\nabla _Yg(\phi X,Z)\xi-(\alpha -1)g(\phi Y,Z)\phi X+(\alpha-1)g(\phi X,Z)\phi Y \\
\nonumber&&-\eta(Z)\phi[X,Y]+g(\phi[X,Y],Z)\xi,
\end{eqnarray}
\begin{eqnarray}
\label{5.10} nor(\overline{\widetilde{R}}(X,Y)Z) &=& \sigma(X,\nabla_YZ)-\sigma(Y,\nabla_XZ)+\nabla_X^\bot \sigma(Y,Z)\\
\nonumber&&-\nabla_Y^\bot \sigma(X,Z)+\eta(Z)\sigma(X,\phi Z)\\
\nonumber&&-\eta(Z)\sigma(Y,\phi Z)-\sigma([X,Y],Z)
\end{eqnarray}
\end{theorem}
\begin{proof}
The Riemannian curvature tensor $\overline{\widetilde{R}}$  on $\widetilde{M}$ with respect to quarter symmetric metric connection is given by
\begin{equation}\label{5.11}
\overline{\widetilde{R}}(X,Y)Z =\overline{\widetilde{\nabla}}_X\overline{\widetilde{\nabla}}_YZ-\overline{\widetilde{\nabla}}_X\overline{\widetilde{\nabla}}_YZ-\overline{\widetilde{\nabla}}_{[X,Y]}Z.
\end{equation}
By using (\ref{3.32}), (\ref{3.36}), (\ref{4.6}), (\ref{5.2}), (\ref{5.4}) and (\ref{5.6a}) in (\ref{5.11}) we get
\begin{eqnarray}
\nonumber\overline{\widetilde{R}}(X,Y)Z &=& \widetilde{R}(X,Y)Z+\sigma(X,\nabla_YZ)+\eta(\nabla_YZ)\phi X-g(\phi X,\nabla_YZ)\xi\\
\label{5.12}&&-A_{\sigma(Y,Z)}X+\nabla_X^\bot \sigma(Y,Z)+\nabla_X\eta(Z)\phi Y+\eta(Z)\nabla_X\phi Z\\
\nonumber &&+\eta(Z)\sigma(X,\phi Z) -\nabla_Xg(\phi Y,Z)\xi-(\alpha-1)g(\phi Y,Z)\phi X-\sigma(Y,\nabla_XZ)\\
\nonumber&&-\eta(\nabla_XZ)\phi Y+g(\phi Y,\nabla_XZ)\xi +A_{(\sigma X,Z)}Y-\nabla_Y^\bot \sigma(X,Z)-\nabla_Y\eta(Z)\phi X\\
\nonumber&&-\eta(Z)\nabla_Y\phi Z-\eta(Z)\sigma(Y,\phi Z)+\nabla_Yg(\phi X,Z)\xi+(\alpha-1)g(\phi X,Z)\phi Y\\
\nonumber&&-\sigma([X,Y],Z)-\eta(Z)\phi[X,Y]+g(\phi[X,Y],Z)\xi.
\end{eqnarray}
Comparing the tangential and normal part of the equation (\ref{5.12}) we get the Gauss and Weingarten formulae as (\ref{5.9}) and (\ref{5.10}).
\end{proof}
\section{Recurrent invariant submanifold of $(LCS)_n$-manifold with respect to quarter symmetric metric connection}
We consider invariant submaniofld of a $(LCS)_n$-manifold when $\sigma$ is recurrent, 2-recurrent,
generalized 2-recurrent and $M$ has parallel third fundamental form with respect to quarter
symmetric metric connection. We write the equations (\ref{3.21}) and (\ref{3.22}) with respect
to quarter symmetric metric connection in the form
\begin{eqnarray}
\label{6.1}(\overline{\widetilde{\nabla}}_X\sigma)(Y,Z) &=& \overline{\nabla}^\bot_X(\sigma(Y,Z))-\sigma(\overline{\nabla}_XY,Z)-\sigma(Y,\overline{\nabla}_XZ) \\
\label{6.2} (\overline{\widetilde{\nabla}}^2\sigma)(Z,W,X,Y) &=&(\overline{\widetilde{\nabla}}_X\overline{\widetilde{\nabla}}_Y\sigma)(Z,W)  \\
\nonumber &=&\overline{\nabla}^\bot_X((\overline{\widetilde{\nabla}}_Y\sigma)(Z,W))-(\overline{\widetilde{\nabla}}_Y\sigma)(\overline{\nabla}_XZ,W)\\
\nonumber&&-(\overline{\widetilde{\nabla}}_X\sigma)(Z,\overline{\nabla}_YW)-(\overline{\widetilde{\nabla}}_{\overline{\nabla}_XY}\sigma)(Z,W).
\end{eqnarray}
\begin{theorem}
  Let $M$ be an invariant submanifold of a $(LCS)_n$-manifold $\widetilde{M}$ with respect to quarter symmetric metric connection. Then $\sigma$ is recurrent with respect to the quarter symmetric metric connection if and only if it is totally geodesic with respect to the Levi-Civita connection, provided $\alpha\neq1$.
\end{theorem}
\begin{proof}
Let $\sigma$ be recurrent with respect to quarter symmetric metric connection. Then from (\ref{2.3}) we get
\begin{equation*}
(\overline{\widetilde{\nabla}}_X\sigma)(Y,Z)=\pi (X) \sigma(Y,Z),
\end{equation*}
where $\pi$ is a 1-form on $M$. By using (\ref{6.1}) and $Z=\xi$ in the above equation we have
\begin{equation}\label{6.3}
\overline{\nabla}^\bot_X(\sigma(Y,\xi))-\sigma(\overline{\nabla}_XY,\xi)-\sigma(Y,\overline{\nabla}_X\xi)=\pi(X)\sigma(Y,\xi).
\end{equation}
  which by virtue of (\ref{3.32}) reduces to
  \begin{equation}\label{6.4}
  -\sigma(\overline{\nabla}_XY,\xi)-\sigma(Y,\overline{\nabla}_X\xi)=0.
  \end{equation}
  Using (\ref{3.31}), (\ref{3.32}), (\ref{3.36}) and (\ref{5.3}) we obtain $(\alpha-1) \ \sigma(X,Y)=0$. So, we get $\sigma(X,Y)=0$, provided $\alpha\neq1$ i.e., $M$ is totally geodesic, provided $\alpha\neq1$. The converse statement is trivial. This proves the theorem.
\end{proof}
\begin{theorem}
Let $M$ be an invariant submanifold of a $(LCS)_n$-manifold $\widetilde{M}$ admitting quarter symmetric metric connection. Then $M$ has parallel third fundamental form with respect to the quarter symmetric metric connection if and only if it is totally geodesic with respect to the Levi-Civita connection, provided $\alpha\neq1$.
\end{theorem}
\begin{proof}
Let $M$ has parallel third fundamental form with respect to quarter symmetric metric connection. Then we have
\begin{equation*}
(\overline{\widetilde{\nabla}}_X\overline{\widetilde{\nabla}}_Y\sigma)(Z,W)=0.
\end{equation*}
Taking $W=\xi$ and and using (\ref{6.2}) in the above equation, we get
\begin{eqnarray}
\label{6.5} 0&=&\overline{\nabla}^\bot_X((\overline{\widetilde{\nabla}}_Y\sigma)(Z,\xi))-(\overline{\widetilde{\nabla}}_Y\sigma)(\overline{\nabla}_XZ,\xi)\\
\nonumber&&-(\overline{\widetilde{\nabla}}_X\sigma)(Z,\overline{\nabla}_Y\xi)-(\overline{\widetilde{\nabla}}_{\overline{\nabla}_XY}\sigma)(Z,\xi).
\end{eqnarray}
By using (\ref{3.32}) and (\ref{6.1}) in (\ref{6.5}) we get
\begin{eqnarray}
\label{6.6}  0&=&-\overline{\nabla}^\bot_X\{\sigma(\overline{\nabla}_YZ,\xi)+\sigma(Z,\overline{\nabla}_Y\xi)\}-\overline{\nabla}^\bot_Y\sigma(\overline{\nabla}_XZ,\xi)\\
\nonumber&& +\sigma(\overline{\nabla}_Y\overline{\nabla}_XZ,\xi)+2\sigma(\overline{\nabla}_XZ,\overline{\nabla}_Y\xi)-\overline{\nabla}^\bot_X\sigma(Z,\overline{\nabla}_Y\xi)\\
\nonumber&&+\sigma(Z,\overline{\nabla}_X\overline{\nabla}_Y\xi)+\sigma(\overline{\nabla}_{\overline{\nabla}_XY}Z,\xi)+\sigma(Z,\overline{\nabla}_{\overline{\nabla}_XY}\xi).
\end{eqnarray}
In view of (\ref{3.31}), (\ref{3.32}), (\ref{3.36}) and (\ref{5.3}) the equation (\ref{6.6}) gives
\begin{eqnarray}
\label{6.7}
0&=&-2(\alpha-1) \overline{\nabla}^\bot_X\sigma(Z,Y)+2(\alpha-1)\sigma(\nabla_XZ,Y)\\
\nonumber&&+2(\alpha-1)\eta(Z)\sigma(X,Y)+(\alpha-1)\sigma(Z,\nabla_X\phi Y)\\
\nonumber&&+(\alpha-1)\sigma(Z,\nabla_XY)+(\alpha-1)\eta(Y)\sigma(Z,X).
\end{eqnarray}
Putting $Z=\xi$ in (\ref{6.7}) and using (\ref{3.9}), (\ref{3.31}), (\ref{3.32}) and (\ref{3.36}) we obtain\\ 2$(\alpha -1)^2 \ \sigma(X,Y)=0$. Therefore, $M$ is totally geodesic provided $\alpha \neq 1$. The converse statement is trivial. This proves the theorem.
\end{proof}
\begin{corollary}
Let $M$ be an invariant submanifold of a $(LCS)_n$-manifold $\widetilde{M}$ admitting quarter symmetric metric connection. Then $\sigma$ is 2-recurrent with respect to the quarter symmetric metric connection if and only if it is totally geodesic with respect to the Levi-Civita connection, provided $\alpha\neq1$.
\end{corollary}
\begin{proof}
Let $\sigma$ be 2-recurrent with respect to quarter symmetric metric connection. Then we have,
\begin{equation*}
  (\overline{\widetilde{\nabla}}_X\overline{\widetilde{\nabla}}_Y\sigma)(Z,W)=\psi(X,Y)\sigma(Z,W).
\end{equation*}
where $\psi$ is a 2-form on $M$. Taking $W=\xi$ and using (\ref{3.32}) in the above equation we get
\begin{equation}\label{6.8}
(\overline{\widetilde{\nabla}}_X\overline{\widetilde{\nabla}}_Y\sigma)(Z,\xi)=0.
\end{equation}
\end{proof}
\begin{theorem}
Let $M$ be an invariant submanifold of a $(LCS)_n$-manifold $\widetilde{M}$ admitting quarter symmetric metric connection. Then $\sigma$ is generalized
2-recurrent with respect to the quarter symmetric metric connection if and only if it is totally geodesic with respect to the Levi-Civita connection, provided $\alpha\neq1$.
\end{theorem}
\begin{proof}
Let $\sigma$ be generalized 2-recurrent with respect to the quarter symmetric metric connection. Then from (\ref{2.5}) we have
\begin{equation}\label{6.12}
(\overline{\widetilde{\nabla}}_X\overline{\widetilde{\nabla}}_Y\sigma)(Z,W)=\psi(X,Y)\sigma(Z,W)+\pi(X)(\overline{\widetilde{\nabla}}_Y\sigma)(Z,W),
\end{equation}
where $\psi$ and $\pi$ are 2-form and 1-form respectively.\\ Taking $W=\xi$ in (\ref{6.12}) and using (\ref{3.32}) we get
\begin{equation}\label{6.13}
(\overline{\widetilde{\nabla}}_X\overline{\widetilde{\nabla}}_Y\sigma)(Z,\xi)=\pi (X)(\overline{\widetilde{\nabla}}_Y\sigma)(Z,\xi).
\end{equation}
In view of (\ref{3.31}), (\ref{3.32}), (\ref{3.36}), (\ref{5.3}), (\ref{6.1}) and (\ref{6.2}) the equation (\ref{6.13}) gives
\begin{eqnarray}
\label{6.14}
-2(\alpha-1) \overline{\nabla}^\bot_X\sigma(Z,Y)+2(\alpha-1)\sigma(\nabla_XZ,Y)\\
\nonumber+2(\alpha-1)\eta(Z)\sigma(X,Y)+(\alpha-1)\sigma(Z,\nabla_X\phi Y)\\
\nonumber+(\alpha-1)\sigma(Z,\nabla_XY)+(\alpha-1)\eta(Y)\sigma(Z,X)=\\
\nonumber-(\alpha-1)\pi(X)\sigma(Y,Z).
\end{eqnarray}

  Putting $Z=\xi$ in (\ref{6.14}) and using (\ref{3.9}), (\ref{3.31}), (\ref{3.32}) and (\ref{3.36}) we obtain \\ 2$(\alpha -1)^2 \ \sigma(X,Y)=0$. Therefore, $M$ is totally geodesic provided $\alpha \neq 1$. The converse statement is trivial. This proves the theorem.
\end{proof}
From Theorem 5.1, Theorem 5.2, Corollary 5.1 and Theorem 5.3, we can state the following:

\begin{theorem}
Let $M$ be an invariant submanifold of a $(LCS)_n$-manifold $\widetilde{M}$ admitting quarter symmetric metric connection. Then the following statements are equivalent:\\ \ (1) $\sigma$ is recurrent\\ \ (2) $\sigma$ is 2-recurrent with $\alpha\neq1$.\\ \ (3) $\sigma$ is generalized 2-recurrent with $\alpha\neq1$.\\ \ (4) $M$ has parallel third fundamental form, with $\alpha\neq1$.\\ \ (5) $M$ is totally geodesic with respect to Levi-Civita connection.
\end{theorem}
\noindent{\bf Acknowledgement:} The first author (S. K. Hui)  gratefully acknowledges to
the SERB (Project No.: EMR/2015/002302), Govt. of India for financial assistance of the work.

\vspace{0.1in}
\noindent S. K. Hui\\
Department of Mathematics, The University of Burdwan, Golapbag, Burdwan -- 713104, West Bengal, India\\
E-mail: shyamal\_hui@yahoo.co.in, skhui@math.buruniv.ac.in\\

\noindent L.-I. Piscoran, North University Center of Baia Mare, Technical University of Cluj Napoca, Department of Mathematics and Computer Science, Victoriei 76, 430122 Baia Mare, Romania\\
E-mail: plaurian@yahoo.com\\

\noindent T. Pal, A. M. J. High School, Mankhamar, Bankura -- 722144, West Bengal, India\\
E-mail: tanumoypalmath@gmail.com

\begin{thebibliography}{20}
\bibitem{ALNI}
Ali, S. and Nivas, R., \emph{On submanifolds immersed in a manifolod with quarter-symmetric connection}, Riv. Mat. Univ. Parma, {\bf 3} (2000), 11--23.
\bibitem{ANIBA}
Anitha, B. S. and Bagewadi, C. S., \emph{Invariant submanifolds of Kenmotsu manifolds admitting quarter symmetric metric connection}, Bull. of Math. Anal. and Appl. {\bf 4} (2012), 99--114.
\bibitem{ANIBA1}
Anitha, B. S. and Bagewadi, C. S., \emph{Invariant submanifolds of Sasakian manifolds}, Differential Integral Equations, {\bf 16(10)} (2003), 1249--1280.
\bibitem{ANIBA2}
Anitha, B. S. and Bagewadi, C. S., \emph{Invariant submanifolds of Sasakian manifolds admitting quarter symmetric metric connection II}, Ilirias J. of Math., {\bf 1} (2012), 1--13.
\bibitem{bejancu}
Bejancu, A. and Papaghuic, N., \emph{Semi-invariant submanifolds of a Sasakian manifold},
An Sti. Univ. ``AL I CUZA'' Iasi, {\bf 27} (1981), 163--170.
\bibitem{DEUD}
De, A. and Uddin, S., \emph{Gauss and Ricci equations in Contact manifolds with a quarter-symmetric metric connection}, Bull. Malays. Math. Sci. Soc.,{\bf 38 (4)} (2015), 1689 --1703.
\bibitem{DESZCZ}
Deszcz, R., \emph{On pseudosymmetric spaces}, Bull. Soc. Belg. Math. Ser A, {\bf 44} (1992), 1--34.
\bibitem{FRSC}
Friedmann, A. and Schouten, J. A., \emph{Uber die geometric derhalbsymmetrischen Ubertragung}, Math. Zeitscr., {\bf 21} (1924), 211--223.
\bibitem{GOLAB}
Golab, S., \emph{On semi-symmetric and quarter symmetric linear connections}, Tensor, N. S., {\bf 29} (1975), 249-254.
\bibitem{HAYDEN}
Hayden H. A., \emph{Subspace of a space with torsion}, Proc. London Math. Soc., {\bf 34} (1932), 27--50.
\bibitem{SKH}
Hui, S. K., \emph{On $\phi$-pseudosymmetries of $(LCS)_n$-manifolds}, Kyungpook Math. J., {\bf 53} (2013), 285--294.
\bibitem{HUI2}
Hui, S. K., \emph{On $\phi$-pseudo symmetric Kenmotsu manifolds with respect to quarter symmetric metric connection},
Applied Sciences, {\bf 15} (2013), 71--84.
\bibitem{HUI1}
Hui, S. K. and Atceken, M., \emph{Contact warped product semi-slant submanifolds of $(LCS)_n$-manifolds},
Acta Univ. Sapientiae Math., {\bf 3} (2011), 212--224.
\bibitem{8}
Matsumoto, K., \emph{On Lorentzian almost paracontact manifolds},
Bull. of Yamagata Univ. Nat. Sci., {\bf 12} (1989), 151--156.
\bibitem{9}
Mihai, I. and Rosca, R., \emph{On Lorentzian para-Sasakian
manifolds}, Classical Analysis, World Scientific Publ., Singapore,
(1992), 155--169.
\bibitem{MISHRA}
Mishra, R. S., and Pandey, S. N., \emph{On quarter-symmetric metric F-connection}, Tensor, N. S., {\bf 34} (1980), 1--7.
\bibitem{MRB}
Mukhopadhyay, S., Roy, A. K. and Barua, B., \emph{Some properties of a quarter-symmetric metric connection on a Riemannian manifold}, Soochow J. Math., {\bf 17} (1991), 205--211.
\bibitem{NIL}
O'Neill, B., \emph{Semi Riemannian geometry with applications to relativity}, Academic Press,
New York, {\bf 1983}.
\bibitem{PRAKAS2}
Prakasha, D. G., \emph{On Ricci $\eta$-recurrent $(LCS)_n$-manifolds}, Acta Univ. Apulencis, {\bf 24} (2010), 109--118.
\bibitem{PRAKAS}
Prakasha, D. G. and Hadimani, B. S., \emph{Some class of Sasakian manifold with respect to quarter symmetrric metric connection}, Int. J. of Math. and its Appl., {\bf 4} (2016), 19--26.

\bibitem{RASTOGI}
Rastogi, S. C., \emph{On quarter-symmetric metric connection}, C. R. Acad. Sci. Sci. Bulgar, {\bf 31} (1978), 811--8147.
\bibitem{ROTER}
Roter, W., \emph{On conformally recurrent Ricci-recurrent manifolds}, Colloq Math., {\bf 46}(1) (1982), 45--57.
\bibitem{SHAIKH2}
Shaikh, A. A., \emph{On Lorentzian almost paracontact manifolds with
a structure of the concircular type}, Kyungpook Math. J., {\bf 43}
(2003), 305--314.
\bibitem{SHAIKH1}
Shaikh, A. A., \emph{Some results on $(LCS)_{n}$-manifolds}, J. Korean Math. Soc.,
{\bf 46} (2009), 449--461.
\bibitem{SHAIKH3}
Shaikh, A. A. and Baishya, K. K., \emph{On concircular structure
spacetimes}, J. Math. Stat., {\bf 1} (2005), 129--132.
\bibitem{SHAIKH4}
Shaikh, A. A. and Baishya, K. K., \emph{On concircular structure
spacetimes II}, American J. Appl. Sci., {\bf 3(4)} (2006),
1790--1794.
\bibitem{SHAIKH5}
Shaikh, A. A., Basu, T. and Eyasmin, S., \emph{On locally $\phi$-symmetric $(LCS)_{n}$-manifolds},
Int. J. Pure Appl. Math., {\bf 41} (2007), 1161--1170.
\bibitem{SHAIKH6}
Shaikh, A. A., Basu, T. and Eyasmin, S., \emph{On the existence of $\phi$-recurrent $(LCS)_{n}$-manifolds},
Extracta Mathematicae, {\bf 23} (2008), 71--83.
\bibitem{SHAIKH7}
Shaikh, A. A. and Binh, T. Q., \emph{On weakly symmetric $(LCS)_{n}$-manifolds},
J. Adv. Math. Studies, {\bf 2} (2009), 75--90.
\bibitem{SHAIKH8}
Shaikh, A. A. and Hui, S. K., \emph{On generalized $\phi$-recurrent
$(LCS)_n$-manifolds}, AIP Conference Proceedings, {\bf 1309} (2010),
419--429.
\bibitem{SHAIKH9}
Shaikh, A. A., Matsuyama, Y and Hui, S. K., \emph{On invariant submanifold of $(LCS)_n$-manifolds},
J. of the Egyptian Math. Soc., {\bf 24} (2016), 263--269.
\bibitem{SIBA}
Siddesha, M. S. and Bagewadi, C. S., \emph{Invariant submanifold of $(k,\mu)$-contact metric manifold admitting
quarter symmetric metric connection}, Int. J. of Math. Trends and Tech. (IJMTT), {\bf 34} (2016), 48--53.
\bibitem{YANO}
Yano, K., \emph{On semi-symmetric metric connections}, Resv. Roumaine Math. Press Apple., {\bf 15}(1970), 1579--1586.
\bibitem{15}
Yano, K., \emph{Concircular geometry I, Concircular transformations}, Proc. Imp. Acad. Tokyo, {\bf 16} (1940), 195--200.
\bibitem{YANOI}
Yano, K. and Imai, T., \emph{Quarter-symmetric metric connections and their curvature tensor}, Tensors, N. S., {\bf 38} (1982), 13--18.
\bibitem{YANO3}
Yano, K. and Kon, M., \emph{Structures on manifolds}, World Scientific publishing, 1984.
\end{thebibliography}
\end{document}